\documentclass[12pt]{article}
\usepackage{geometry}
\usepackage{fancyhdr}
\usepackage{amsmath,amsthm,amssymb}
\usepackage{dsfont}
\usepackage{hyperref}
\usepackage{amsfonts}
\usepackage{verbatim}
\usepackage{tabularx}
\usepackage{enumerate}
\usepackage{color}
\usepackage[normalem]{ulem}

\newtheorem{theorem}{Theorem}[section]
\newtheorem{corollary}[theorem]{Corollary}
\newtheorem{lemma}[theorem]{Lemma}

\newtheorem{proposition}[theorem]{Proposition}

\theoremstyle{remark}\newtheorem{definition}[theorem]{Definition}
\theoremstyle{remark}\newtheorem{remark}[theorem]{Remark}

\newcommand{\C}{\mathbb{C}}
\newcommand{\D}{\mathbb{D}}
\newcommand{\E}{\mathbb{E}}

\newcommand{\Q}{\mathbb{Q}}

\newcommand{\R}{\mathbb{R}}
\renewcommand{\P}{\mathbb{P}}

\renewcommand{\Re}{\mathrm{Re}}

\newcommand{\wt}{\widetilde}

\newcommand{\eps}{\varepsilon}

\newcommand{\1}{\mathbf{1}}

\newcommand{\DMSsph}[1]{ \mu_{\mathrm{DMS}}^{#1}  }
\newcommand{\DKRVsph}[1]{ \mu_{\mathrm{DKRV}}^{#1}  }

\def\cQ{\mathcal{Q}}

\DeclareMathOperator{\Var}{Var}

\DeclareMathOperator{\dist}{dist}

\begin{document}

	\author{\begin{tabular}{c}Juhan Aru
			\thanks{juhan.aru@math.ethz.ch. Department of Mathematics, ETH Z\"urich, R\"amistr. 101, 8092 Z\"urich, Switzerland
			}	
		\end{tabular}
		\and \begin{tabular}{c}
			Yichao Huang  \thanks{yichao.huang@ens.fr. Sorbonne Universities, UPMC Univ Paris 06 \& Department of Mathematics, ENS Paris, CNRS, PSL Research University, 45 Rue d'Ulm, 75005 Paris, France
			}
		\end{tabular}\and 
		\begin{tabular}{c}Xin Sun
			\thanks{xinsun89@math.mit.edu. Department of Mathematics, MIT, Cambridge, MA, 02139, U.S.
			} 
		\end{tabular}}
		\title{Two perspectives of the 2D unit area quantum sphere and their equivalence}
		\date{}
		\maketitle
		\begin{abstract}
			2D Liouville quantum gravity (LQG) is used as a toy model for 4D quantum gravity and is the theory of world-sheet in string theory. Recently there has been growing interest in studying LQG in the realm of probability theory: David, Kupiainen, Rhodes, Vargas (2014) and Duplantier, Miller, Sheffield (2014) both provide a probabilistic perspective of the LQG on the 2D sphere. In particular, in each of them one may find a definition of the so-called unit area quantum sphere. We examine these two perspectives and prove their equivalence by showing that the respective unit area quantum spheres are the same. This is done by considering a unified limiting procedure for defining both objects.
		\end{abstract}
		
		\section{Introduction}\label{sec:intro}
		2D Liouville quantum gravity (LQG) was first introduced by Polyakov in \cite{POLYAKOV1} as a framework for integrating over surfaces. On the one hand, this integration over surfaces can be used to describe the time-evolution of bosonic strings and on the other hand, it provides a model for a random metric with a fixed topology, i.e. for quantum gravity \cite{SEIBERG}. Whereas the theory for surfaces of non-trivial moduli remains to be understood, the basic theory and the constructions for the 2D sphere are well-known in the physics literature \cite{NAK}. In the realm of rigorous probability theory, however, even the understanding of the simplest case, i.e. the theory of the random metric on the 2D sphere is relatively recent. 
		
		Liouville quantum gravity on the sphere is the limiting object of natural discrete random planar map models on the sphere. This provides one way to define the limiting object as a random measure endowed metric space \cite{LeGall,Miermont}, however the conformal structure is lacking in this framework. 
		Indeed, as in the case of 2D Riemannian manifolds, this metric, although highly non-smooth, should nevertheless determine a conformal structure such that both the metric and the measure are described using a 2D Gaussian free field defined on this conformal structure \cite{POLYAKOV1, SEIBERG}. For now (except for the recent progress in the case of $\gamma = \sqrt{8/3}$, see \cite{QLE-TBM1,QLE-TBM2,QLE-TBM3}) one can only define the corresponding volume form, that takes roughly the form `$e^{\gamma h}dz$'  where $h$ is a 2D Gaussian free field type of distribution and $dz$ the volume element.
		In fact, two different constructions of the volume form corresponding to the unit area quantum sphere have recently appeared in the probability literature. 
		
		\cite{DKRV} rigorously constructed the Liouville quantum field theory on the sphere by proving that the partition function of the theory is indeed well-defined. As a consequence, \cite{DKRV} also provides an exact formulation of the Liouville measure on the sphere conditioned to have unit volume. Later \cite{disk,torus} generalized the construction to the disk and to the torus and in \cite{critical}, the authors constructed the Liouville quantum gravity on the sphere in a certain critical situation -- when the so-called Seiberg bound is saturated. In \cite{KRV15} the authors further verify that the conformal Ward and BPZ identities for Liouville quantum field theory can be derived in the probabilistic framework. This is an important step in their project of unifying the path integral approach and the conformal bootstrap approach of Liouville conformal field theory. An important feature of this approach is the use of Gaussian multiplicative chaos theory and its strong link to the original physics literature.
		
		Another approach stems from \cite{Zipper}, where the author suggested a limiting procedure involving the Gaussian free field.  Following up this work, in \cite{mating}, the authors provided a more concrete construction via Bessel processes and showed that it is equivalent to the limiting procedure suggested in \cite{Zipper}. \cite{mating} also rigorously constructed objects like quantum disks, quantum wedges and  quantum cones. In \cite{mating2}, the authors provide further constructions of the quantum sphere. In particular, they show that the unit area quantum sphere can be obtained from mating a pair of correlated continuum trees of finite diameter, which are given by space filling SLE curves. This exemplifies an important aspect of their approach: the interplay between Gaussian free field and Schramm-Loewner evolution, already known to be linked to conformal field theories \cite{CARDY}.
		
		The \cite{DKRV} approach considers the so-called Liouville measure directly in the space of random measures whereas in the \cite{mating} approach one defines the quantum sphere as an equivalence classes of random measures. On one hand, the \cite{DKRV} approach is more explicit. On the other hand, the notion of equivalence class enables \cite{mating} to work with one or two marked points whereas the framework of \cite{DKRV} is restricted to 3 or more points.
		
		Both approaches provide evidence that they give the correct scaling limit of the random planar maps weighted by critical statistical mechanics models with 3 uniformly chosen marked points, (see \cite[Section 5.3]{DKRV} and \cite[Section 1.2]{mating2},\cite{loop2}). Following the notions preferred by the original authors, the candidate via the \cite{DKRV} approach is called the \emph{unit volume Liouville measure with three insertion points of weight $\gamma$}. The candidate via the \cite{mating} approach is called the \emph{unit area quantum sphere with three marked points}. A natural question is whether the two constructions actually agree. The main aim of this article is to give an affirmative answer to this question (Theorem~\ref{thm:main}). On the way we also revisit the two perspectives in some detail and provide a unified framework to work with both of them.
		
		We finally remark that there is yet a third mathematical approach to defining a unit area quantum sphere in the case when $\gamma=\sqrt{8/3}$ (the so-called pure gravity). In this case, the unit area quantum sphere can be defined as the scaling limit of a large class of random planar maps including uniform triangulations, uniform quadrangulations etc. Indeed, it has been shown that these random planar maps converge as metric spaces to the Brownian map in the Gromov-Hausdorff topology \cite{LeGall,Miermont}. Moreover, \cite{QLE,QLE-TBM1} recently announced that there is a canonical way of putting a metric on the unit area quantum sphere with $\gamma=\sqrt{8/3}$ via the quantum Loewner evolution to obtain the Brownian map, thus proving the equivalence between the Brownian map and the unit area quantum sphere defined in \cite{mating}. Our result  relates \cite{DKRV} to this body of works, by showing that the unit area quantum sphere can be constructed also via their approach.
		
		\bigskip
		
		\noindent \textbf{Acknowledgements}
		We are grateful to Scott Sheffield for suggesting this problem and helpful discussions. We are also thankful to Fran\c cois David, Bertrand Duplantier, Jason Miller, R\'emi Rhodes and Vincent Vargas for useful discussions. We thank the anonymous referee and Bertrand Duplantier for their careful reading and many helpful comments on the manuscript. The authors would like to thank the Isaac Newton Institute for Mathematical Sciences, Cambridge, for support and hospitality during the program \emph{Random Geometry} where most of the work on this paper was undertaken. J. Aru acknowledges support of the SNF grant SNF-155922, Y. Huang would like to thank the support from the Cambridge-PSL French embassy fund and X. Sun was partially supported by NSF grant DMS-1209044.
		
		\subsection{Outline and main results}
		We will revisit  the two constructions of the unit area quantum sphere in detail in Section~\ref{sec:two-pers}. Here we first give a brief description in order to state our main theorem.
		
		In \cite{DKRV}, given an integer $k\ge 3$, $z_1,\dotsm,z_k\in \C\cup\{\infty\}$ and $\alpha_1\dotsm,\alpha_k\in \R$ with $\alpha_1,\dotsm,\alpha_k$ satisfying certain bounds, one can use the Gaussian multiplicative chaos theory to construct the so-called Liouville measure on the sphere with log singularity $\alpha_i$ at $z_i$. Then by conditioning on the quantum area to be 1, one obtains a Liouville measure with unit volume.  We denote the law of unit volume Liouville measures obtained in this way by $\DKRVsph{\alpha_1,\dotsm,\alpha_k}$. The detailed construction will be provided in Section~\ref{Description:DKRV}.
		
		In \cite{mating}, the construction of the unit area quantum sphere  is based on the notion of quantum surfaces --- an equivalence class of random distributions with a number of marked points. In Section~\ref{Description:DMS14} we revisit the Bessel process construction of the unit area quantum sphere with 2 marked points, whose law is denoted by $\DMSsph{2}$.  The unit area quantum sphere with  $k\ge 3$  marked points (denoted by  $\DMSsph{k}$) can be obtained by first sampling according to $\DMSsph{2}$ and then sampling $k-2$ points according to the quantum area independently of each other. 
		
		The following is the main theorem of the paper:
		\begin{theorem}\label{thm:main} 
			For $\gamma\in (0,2)$,	
			$$\DMSsph{3}=\DKRVsph{\gamma,\gamma,\gamma}.$$ 
			More precisely, if we embed  $\DMSsph{3}$ such that the three marked points are fixed at $z_1,z_2,z_3\in \C\cup \{\infty\}$, then it has the same distribution as $\DKRVsph{\gamma,\gamma,\gamma}$ with marked points of weight $\gamma,\gamma,\gamma$ at $z_1, z_2, z_3$.
		\end{theorem}
		\begin{remark}
			Whereas we state the equivalence in terms of measures, it also holds in terms of underlying fields and thus underlying quantum surfaces. Indeed, in \cite{measure-field}, it is shown that a Gaussian free field type field and the Liouville measure it induces determine each other. One can check that this result also applies in the current context.
		\end{remark}
		\begin{remark}
			Notice that
			$\DMSsph{4}\neq\DKRVsph{\gamma,\gamma,\gamma,\gamma}.$ In fact, according to the definition, the cross ratio of the 4 marked points of $\DMSsph{4}$ is a non-trivial random variable. On the side of $\DKRVsph{\gamma,\gamma,\gamma,\gamma}$, however, the cross ratio is simply the cross ratio of the four given points $z_1,z_2,z_3, z_4$. For the same reason Theorem~\ref{thm:main} is not true for any $k>3$. The value $k=3$ is special because on the one hand it fixes a conformal structure on the sphere, but on the other hand all spheres with three marked points remain conformally equivalent. 
			
			On the one hand, the area measure of random planar quadrangulations with $n$ faces, with each face carrying area $1/n$ and with four marked points chosen independently from the area measure, should conjecturally converge to $\DMSsph{4}$ in the limit $n \rightarrow \infty$. On the other hand, if one first embeds random planar quadrangulations with $n$ faces, with each face carrying area $1/n$ and with three marked points, to the Riemann sphere with three marked points $z_1, z_2, z_3$ using say circle-packing embedding, and then adds a suitable singularity at a fixed point $z_4$ (see \cite{DKRV} section 5.3), one expects to obtain in the $n \rightarrow \infty$ limit the random  measure $\DKRVsph{\gamma,\gamma,\gamma,\gamma}$ with four marked points at $z_1, z_2, z_3, z_4$.
		\end{remark}
		
		The idea of the proof is to approximate $\DKRVsph{\gamma,\gamma,\gamma}$ and $\DMSsph{3}$ using essentially the same limiting procedure. In Section~\ref{sec:3pt} we will provide an approximation  scheme for $\DKRVsph{\gamma,\gamma,\gamma}$ which is very close to the one in \cite[Section 5.3]{mating}. In Section~\ref{sec:equi} we will show that a slight perturbation of the scheme in Section~\ref{sec:3pt} leads to both  $\DKRVsph{\gamma,\gamma,\gamma}$ and $\DMSsph{3}$, thus concluding the proof.
		
		\section{Two perspectives of the unit area quantum sphere}\label{sec:two-pers}
		In this section we first provide necessary background on the whole-plane Gaussian free field and then describe the two perspectives in some more detail.
		
		\subsection{The whole plane Gaussian free field}\label{subsec:whole}
		As in \cite{GFF}, the whole plane GFF $h$ is random distribution module an additive constant satisfying
		\[
		\Var[(h,\phi)]=-\int  \phi(x)\log|x-y|\phi(y)dxdy, \quad \forall \phi\in C^\infty_0(\R^2)\; \textrm{s.t.}\;\int\phi(x)dx=0.
		\]
		In other words, it is a Gaussian process only indexed by zero mean test functions. 
		
		One can also make $h$ into an honest random distribution (i.e. a scalar field) by pinning it: namely, we choose some finite measure $\rho(z)dz$ on $\C\cup\{\infty\}$ satisfying $\int\int|\rho(x)\log|x-y| \rho(y)|dxdy < \infty$ and set the average $(h,\rho)=\int h(z)\rho(z)dz$ to be zero, i.e. to pin the field using $\rho$. We sometimes call $\rho(z)dz$ the background measure. Then all the other averages $(h, \rho')$ are defined by choosing $c>0$ such that $c\rho$ has the same mass as $\rho'$ and setting $(h, \rho') := (h, \rho' - c\rho)$.
		
		If $\rho$ has unit mass, then the Green's function of such a field is given by \cite{DKRV}:
		\begin{equation}\label{def:green}
		G_{\rho}(z,w)=\log \frac{1}{|z-w|} - m_\rho(\log \frac{1}{|w-\cdot|}) - m_\rho(\log \frac{1}{|z-\cdot|}) + \theta_\rho,
		\end{equation}
		where we set $m_\rho(f) := \int_{\R^2} f(z)\rho(z)dz$ and $\theta_\rho = -\int\int_{\R^2 \times \R^2}\rho(z)\log|z-w|\rho(w)dwdz$.
		
		We denote the whole-plane GFF pinned using $\rho$ by $h_\rho$, its probability measure by $\P_\rho$ and the corresponding expectation operator by $\E^{\P_\rho}$. For a different background measure the scalar whole plane GFF differs by a random constant. No choice of background measure makes the field truly conformal invariant -- conformal transformations change the pinning of the field. \cite{DKRV} circumvents this by tensoring the whole plane GFF with the Lebesgue measure on $\R$ to obtain a Mobius invariant, but infinite measure. We will see the precise meaning in the next subsection. 
		
		One of the natural choices for the background measure is the normalized area measure of the spherical metric $\hat g(x)=\pi^{-1}(1+|x|^2)^{-2}$ of total mass $1$. More explicitly, we define $h_{\hat g}$ to be the random distribution with covariance given by \cite[Equation (2.12)]{DKRV}:
		\[
		G_{\hat g}(x,y):=\log\frac{1}{|x-y|}-\frac{1}{4}(\log \hat g(x)+\log\hat{g}(y) )-\frac{1}{2}.
		\]
		
		We refer to \cite{IG-IV} for more information on whole plane Gaussian free field. For later purpose we record \cite[Proposition 2.10]{IG-IV} stating that one can approximate the whole-plane GFF using Dirichlet free fields on a sequence of growing domains:
		\begin{proposition}\label{prop:whole-plane}
			Suppose $D_n$ is a sequence of growing domains with harmonically non-trivial boundary containing $0$ s.t. $\dist(0,\partial D_n) \rightarrow \infty$. In each $D_n$ we define a Dirichlet GFF $h_n$ with uniformly bounded boundary values. Then on any fixed bounded domain $D$, the restrictions of $h_n$ to $D$ converge in total variational distance to the restriction of the whole plane GFF to $D$, seen as a distribution modulo additive constant.
		\end{proposition}
		
		\subsection{David-Kupiainen-Rhodes-Vargas' approach}\label{Description:DKRV}
		In this section we provide a detailed description of the construction of the unit volume Liouville measure in \cite{DKRV}. We identify the sphere $\mathbb S^2$ with $\C\cup \{\infty\}$ via the stereographic projection map. Before going into details of the construction, we remark that in order to understand the Definition~\ref{def:DKRV-def} of the unit volume Liouville measure, one only needs \eqref{eq:Liouville-field}, \eqref{eq:zgamma}, \eqref{eq:new-Seiberg} and the definition of the Gaussian multiplicative chaos measure, as explained between \eqref{eq:partition-function} and \eqref{eq:zgamma}.
		
		For the spherical metric $\hat g$, the first step in constructing the Liouville field on $\C\cup\{\infty\}$ with marked points at $z_1,\dots,z_k$ and log-singularities $\alpha_i$ at $z_i$ is to consider the whole-plane GFF with a shift term corresponding to the curvature of the metric and additional singularities corresponding to the marked points:
		
		\begin{equation}\label{eq:Liouville-field}
		h_L(z)=h_{\hat g}(z)+\frac{Q}{2}\log \hat g(z)+ \sum_{i=1}^{k}\alpha_iG_{\hat g}(z,z_i).
		\end{equation}
		Here $Q = 2/\gamma + \gamma /2$ with $\gamma \in (0,2)$ a fixed parameter. Now, given the so-called cosmological constant $\hat\mu>0$ and the parameter $\gamma \in (0,2)$, we define the partition function for the Liouville field, motivated by the physics literature. For any bounded continuous functional $F$ acting on $H^{-1}(\C)$ we set:
		\begin{equation}\label{eq:partition-function}
		\Pi^{(z_i,\alpha_i)_i}_{\gamma,\hat\mu}  (F)=e^{C(z)} \prod_{i}\hat g(z_i)^{\Delta_{\alpha_i}}\int_{\R} e^{sc} \E^{\P_{\hat g}}[F(c+h_L)\exp(-\hat\mu e^{\gamma c}\mu_{h_L}(\C) ) ]dc.
		\end{equation}
		Here, for a fixed parameter $\gamma$, $\mu_{h_L}(\cdot)$ denotes the Gaussian multiplicative chaos measure constructed from the log-correlated field $h_L(z)$. Shortly, it is the limit in probability of the measures $\mu_{h_L}^\eps(\cdot)$ defined by setting for any Borel $A \subset \C$, $\mu_{h_L}^\eps(A) = \int_A \eps^{\frac{\gamma^2}{2}}\exp(\gamma h_L^\eps(z))dz$, where $h_L^\eps$ is, for example, the circle-average approximation of $h_L$. See e.g. \cite{DKRV}) for more details. Moreover, we set
		\begin{align}\label{eq:zgamma}
		& s=\sum_{i=1}^k \alpha_i -2Q, \\
		\label{eq:mugamma}
		& C(z)=\frac{1}{2}\sum_{i\neq j} \alpha_i\alpha_j G_{\hat g}(z_i,z_j) +\frac{\theta_{\hat g}+\log 2}{2}\sum_i \alpha^2_i,\nonumber\\
		& \Delta_\alpha= \frac{\alpha}{2}(Q-\frac{\alpha}{2})\nonumber.
		\end{align}
		Finally, the integration over $c$ corresponds to the tensoring of the free field with the Lebesgue measure mentioned above.
		
		In \cite[Theorem 3.2]{DKRV} it is proved that the partition function \eqref{eq:partition-function} is non-trivial, finite and can be approximated by a circle-average regularization procedure if and only if $\alpha_1,\dotsm,\alpha_k\in \R$ satisfy the following Seiberg bounds:
		\begin{align*}
		&\sum^{k}_{i=1}\alpha_i >2Q\quad\text{(first Seiberg bound);}\\
		&\alpha_i<Q,\;\forall i \quad\text{(second Seiberg bound).}
		\end{align*}
		
		We remark that the two Seiberg bounds yield $k\ge 3$. This suits well with the need of three marked points for fixing the conformal structure of the sphere. 
		
		Under these constraints the partition function gives rise to a probability measure on $H^{-1}(\C)$, called the \emph{Liouville field}. Its law is denoted by $\P^{\gamma,\hat\mu}_{(z_i,\alpha_i,\hat g)}$ and it is described by setting:
		
		\begin{equation}\label{eq:field}
		\E^{\gamma,\hat\mu}_{(z_i,\alpha_i,\hat g)}[F] = \frac{\Pi^{(z_i,\alpha_i)_i}_{\gamma,\hat\mu}  (F)}{\Pi^{(z_i,\alpha_i)_i}_{\gamma,\hat\mu}  (1)}.
		\end{equation}for all bounded continuous functionals $F$ on $H^{-1}(\C)$.
		
		Via this partition function and again using the theory of Gaussian multiplicative chaos, the Liouville field also induces a corresponding Liouville measure $M(\cdot)$ on the sphere by specifying for all bounded continuous functional $F$ on $\R^n_+$ and all Borel sets $A_1,\dotsm,A_n$ in $\C$:
		\begin{align*}
		&\E^{\gamma,\hat\mu}_{(z_i,\alpha_i,\hat g)}[F(M(A_1),\dotsm,M(A_n)]\\
		=&\frac{
			\int_{\R} e^{sc} \E^{\P_{\hat g}}[F(e^{\gamma c}\mu_{h_L}(A_1),\dotsm,e^{\gamma c}\mu_{h_L}(A_n))\exp(-\hat\mu e^{\gamma c}\mu_{h_L}(\C) ) ]dc
		}{\int_{\R} e^{sc} \E^{\P_{\hat g}}[\exp(-\hat\mu e^{\gamma c}\mu_{h_L}(\C) ) ]dc
	}.
	\end{align*}
	Using a change of variables $e^{\gamma c}\mu_{h_L}(\C)=y$ we can rewrite this as follows:
	\begin{align*}
	&\E^{\gamma,\hat\mu}_{(z_i,\alpha_i,\hat g)}[F(M(A_1),\dotsm,M(A_n))]\\
	= &\frac{
		\int_0^\infty \E^{\P_{\hat g}}\left[F(y\frac{\mu_{h_L}(A_1)}{\mu_{h_L}(\C)},\dotsm,y\frac{\mu_{h_L}(A_n)}{\mu_{h_L}(\C)})\mu_{h_L}(\C)^{-\frac{s}{\gamma}} \right] e^{-\hat\mu y}y^{\frac{s}{\gamma}-1}dy
	}{\hat\mu^{\frac{s}{\gamma}} \Gamma(\frac{s}{\gamma}) \E^{\P_{\hat g}}\left[\mu_{h_L}(\C)^{-\frac{s}{\gamma}}\right]
}.\end{align*}
Thus the Liouville measure on the sphere can be in fact described as follows:
\begin{enumerate}
	\item $M(\C)$ is independent of the normalized measure $\bar M(\cdot):=M(\cdot)/M(\C)$;
	\item $M(\C)$ is distributed as Gamma distribution $\Gamma(\frac{s}{\gamma},\hat\mu)$;
	\item The normalized measure $\bar M(\cdot)$ can be sampled by the normalized measure of $\mu_{h_L}$  re-weighted by $\mu_{h_L}(\C)^{-\frac{s}{\gamma}}$. Namely, write $\bar\mu_{h_L}(\cdot):= \mu_{h_L}(\cdot) /\mu_{h_L}(\C)$ then
	\begin{equation}\label{eq:normalized-measure}
	\E^{\gamma,\hat\mu}_{(z_i,\alpha_i,\hat g)}[F(\bar M(\cdot))] 
	= \frac{\E^{\P_{\hat g}}\left[F(\bar \mu_{h_L}(\cdot)) \mu_{h_L}(\C)^{-\frac{s}{\gamma}}\right]
	}{\E^{\P_{\hat g}}\left[\mu_{h_L}(\C)^{-\frac{s}{\gamma}}\right]}.
	\end{equation}
\end{enumerate}
Notice that from this explicit structural description of the Liouville field, we see that the cosmological constant $\hat\mu$ only effects the total volume of the Liouville measure of the sphere, while the normalized measure is independent of $\hat\mu$. It is natural to then take the reweighted normalized measure $\bar \mu_{h_L}(\cdot)$ as the definition of the unit volume Liouville measure with specific log-singularities. As stated in
\cite[Lemma 3.10]{DKRV}, for (\ref{eq:normalized-measure}) to make sense, we only need (see Lemma~\ref{lem:moments}):
\begin{equation}\label{eq:new-Seiberg}
Q-\frac{\sum_{i=1}^{k}\alpha_i}{2}<\frac{2}{\gamma}\wedge \min_i(Q-\alpha_i)\quad \textrm{and}\quad \alpha_i<Q, \,\forall i.
\end{equation}
Now we are ready to give a formal definition of $\DKRVsph{\alpha_1,\cdot,\alpha_k}$. The expression for the Liouville field comes from the same change of variables as above:

\begin{definition}\label{def:DKRV-def}
	Fix $\gamma \in (0,2)$. Let $k\ge 3$, $ z_1,\dotsm,z_k $ be distinct points in $\C\cup \{\infty\}$ and $\alpha_1,\dotsm,\alpha_k\in \R$ satisfying \eqref{eq:new-Seiberg}. Let $h_L$ be defined as in \eqref{eq:Liouville-field}, $s$ as in \eqref{eq:zgamma} and $\mu_{h_L}(\cdot)$ be the Gaussian multiplicative chaos measure corresponding to $h_L$.
	
	Then the unit volume Liouville field on the sphere with marked points $(z_1,\dotsm,z_k)$ and log-singularities $(\alpha_1,\dotsm,\alpha_k)$ is a random field distributed as 
	\[
	h_L-\gamma^{-1}\log \mu_{h_L}(\C)
	\]under the re-weighted measure $$d\widehat{\P}^{(\alpha_i,z_i)_i}:=\E^{\P_{\hat g}}\left[\mu_{h_L}(\C)^{\frac{-s}{\gamma}}\right]^{-1} \mu_{h_L}(\C)^{\frac{-s}{\gamma}} d\P_{\hat g}.$$ Correspondingly, the unit volume Liouville measure, denoted hereafter by $\DKRVsph{\alpha_1,\cdot,\alpha_k}$, is the measure distributed as $\bar \mu_{h_L}(\cdot)$ under $\widehat{\P}^{(\alpha_i,z_i)_i}$.
\end{definition}

Notice that we did not include $\hat g$ in the notation of $\widehat{\P}^{(\alpha_i,z_i)_i}$. This is due to Theorem \ref{thm:iofm} which proves that the unit volume Liouville measure does not depend on the choice of normalization. We will be interested in the special case of $k=3$ and $\alpha_i = \gamma$ for all $i =1,2,3$. Note that in this case \eqref{eq:new-Seiberg} is satisfied for all $\gamma\in (0,2)$.

\subsubsection{Properties of the Liouville measure: moments, covariance under Mobius transforms and independence of the background measure}\label{sec:background}
In this section we record two invariance properties for the unit volume Liouville measure $\DKRVsph{\alpha_1,\dotsm,\alpha_k}$, and a criterion for the existence of moments of the Liouville measure.

First, we claim Definition~\ref{def:DKRV-def} is independent of the background measure, justifying our notation $\widehat{\P}^{(\alpha_i,z_i)_i}$ in that definition. Indeed, in Definition~\ref{def:DKRV-def} the only role of the background measure is to fix the whole-plane GFF. We will show in Theorem~\ref{thm:iofm} below that this choice does not matter. This corresponds to the Weyl anomaly described in \cite[Lemma~3.11.2)]{DKRV}.

Consider a probability measure $\rho(z)dz$ and let $h_\rho$ be the whole-plane GFF normalized such that $(h,\rho) = 0$. As above, denote by $m_\rho(f) := \int_{\R^2} f(z)\rho(z)dz$ and set 
\begin{equation}\label{eq:exactdefoffield}
h_{L(\rho)}(z)=h_{\rho}(z)+2Qm_\rho(\log\frac{1}{|z-\cdot|})+\sum\limits_{i=1}^{k}\alpha_i G_{\rho}(z,z_i),
\end{equation}
where $z_i$ are marked points with $\log$-singularities $\alpha_i$, and as in \cite{DKRV}, $G_{\rho}$ is the Green function associated with the background measure $\rho$ as defined in \eqref{def:green}.

Notice that for the normalized spherical metric $\hat{g}$, the field $h_{L(\hat{g})}$ is the same as $h_L$ defined above. Finally, write $\mu_{h_{L(\rho)}}$ as the volume measure corresponding to the field $h_{L(\rho)}$.
\begin{theorem}\label{thm:iofm}
	Take any $\rho_1, \rho_2$ as above and let $s=\sum\limits_{i=1}^{k}\alpha_i-2Q$. The fields
	\begin{equation*}
	h_{L(\rho_1)}-\gamma^{-1}\log \mu_{h_{L(\rho_1)}}(\mathbb{C}) \quad\textrm{under the measure}\quad c\mu_{h_{L(\rho_1)}}(\mathbb{C})^{-\frac{s}{\gamma}}d\P_{\rho_1}
	\end{equation*}
	and
	\begin{equation*}
	h_{L(\rho_2)}-\gamma^{-1}\log \mu_{h_{L(\rho_2)}}(\mathbb{C}) \quad\textrm{under the measure}\quad c\mu_{h_{L(\rho_2)}}(\mathbb{C})^{-\frac{s}{\gamma}}d\P_{\rho_2}
	\end{equation*}
	are equal in law.
	Here the constants $c$ are chosen to make the above measures probability measures.
\end{theorem}
\begin{proof}
	To begin with, recall that there is an identification between the whole-plane GFF pinned using some background measure $\rho$ and the whole-plane GFF seen as a modulo-constant distribution: the passage from the latter to the former was described in Subsection \ref{subsec:whole} and to pass from former to the latter, we just test the GFF against zero mean test functions. Using this identification, it suffices to prove the theorem using the probability measure $d\P$ induced by the whole-plane GFF on the modulo-constant distributions. We write $(h, \phi)$ for this modulo constant field acting on zero mean distributions $\phi$.
	Now, observe that under the measure $d\P$, we have that almost surely $h_{\rho_1} = h_{\rho_2}-(h,\rho_1 - \rho_2)$.
	Write also 
	\[g_{\rho_i}=h_{L(\rho_i)}-h_{\rho_i},\quad i=1,2.\]
	Then in particular
	\[h_{L(\rho_1)} = h_{\rho_2}+g_{\rho_1} - (h,\rho_1 - \rho_2).\]
	As for $i = 1,2$ the quantity $h_{L(\rho_i)}-\gamma^{-1}\log \mu_{h_{L(\rho_i)}} (\C)$ remains unchanged after adding a constant to the field, we can write:
	\[h_{L(\rho_1)}-\gamma^{-1}\log \mu_{h_{L(\rho_1)}}(\mathbb{C}) = h_{\rho_2}+g_{\rho_1}-\gamma^{-1}\log \mu_{h_{\rho_2}+g_{\rho_1}}(\mathbb{C}).\]
	In addition, we also have that
	\[\mu_{h_L(\rho_1)}(\C)^{-\frac{s}{\gamma}}=\mu_{h_{\rho_2}+g_{\rho_1}}(\C)^{-\frac{s}{\gamma}}\exp(s(h, \rho_1 - \rho_2)).\]
	By the Cameron-Martin theorem for the modulo-additive constant Gaussian free field (see e.g. \cite{SJ}), reweighing $d\P$ by $\exp(s(h, \rho_1 - \rho_2))$ induces for any zero mean test function $\phi$ a drift equal to 
	$$-s\int_{\R^2}\int_{\R^2} \phi(z)\log |z-w|(\rho_1-\rho_2)(w)dzdw$$
	on $(h, \phi)$. 
	Now as the difference of Green's functions $G_{\rho_1}(z,z_i) - G_{\rho_2}(z,z_i)$ satisfies 
	\[\Delta (G_{\rho_1}(z,z_i) - G_{\rho_2}(z,z_i)) =\rho_1-\rho_2, \quad \textrm{for all}\; 1\le i\le k.\]
	Finally, one can verify that
	$$-s\int_{\R^2}\phi(z)\log |z-w|(\rho_1-\rho_2)(w)dw = g_{\rho_2}-g_{\rho_1}+C$$
	for some absolute constant $C$.	Now Theorem~\ref{thm:iofm} follows.
\end{proof}

Second, the following theorem says that the unit volume Liouville measure transforms covariantly under Mobius transforms:
\begin{theorem}\label{thm:coordinate-change}
	\cite[Theorem 3.7]{DKRV}
	
	Let $\psi$ be a Mobius transform of the sphere. The law of the unit volume Liouville field $\phi$ under $\widehat{\P}^{(\alpha_i,z_i)_i}$ is the same as that of $\phi\circ \psi+Q\log|\psi'|$ under $\widehat{\P}^{(\alpha_i,\psi(z_i))_i}$. The area measure under $\widehat{\P}^{(\alpha_i,\psi(z_i))_i}$ has the same law as the pushforward of the area measure under $\widehat{\P}^{(\alpha_i,z_i)_i}$.
\end{theorem}

Finally, for the moment estimate that we use later on, we first recall a slightly modified formulation of Lemma 3.10 in \cite{DKRV}:

\begin{lemma}\label{lem:moments}
	Fix $\gamma \in (0,2)$ and let $\rho = \hat{g}$, i.e. take the spherical measure as the pinning measure. Consider the field $h_{L(\rho)}$ as in \eqref{eq:exactdefoffield} under the law $\P_\rho$. Then for any $q \in \R$ with $q < \frac{4}{\gamma^2} \wedge \min_i \frac{2}{\gamma}(Q - \alpha_i)$ we have that $\E^{\P_{\rho}}[\mu_{h_{L(\rho)}}(\C)^q] < \infty.$ 
\end{lemma}

We will need the same estimate in the case where $\rho = \mathfrak c$, where we denote by $\mathfrak c$ the uniform probability measure on the unit circle.

\begin{corollary}\label{cor:moments}
	The lemma above holds holds for $\rho = \mathfrak c$.
\end{corollary}

\begin{proof}
	Notice that
	\begin{equation}
	\left|m_{\hat{g}}(\log\frac{1}{|z-\cdot|})-m_{\mathfrak c}(\log\frac{1}{|z-\cdot|})\right|=\left|\frac{1}{4}\log\hat{g}(z)+\log(|z|\vee 1)\right|<C
	\end{equation}
	where $C$ is some finite constant independent of $z$. 
	
	It follows from the definition of the Green's functions \eqref{def:green} that $G_{\hat{g}}$ and $G_{\mathfrak c}$ differ at most by some finite constant. Moreover, as remarked above, the difference between $h_{\hat{g}}$ and $h_{\mathfrak c}$ is a Gaussian of finite variance. Thus from \eqref{eq:exactdefoffield} we see that the difference between $h_{L(\hat{g})}$ and $h_{L(\mathfrak c)}$ is bounded by some Gaussian with bounded mean and variance and we conclude.
\end{proof}

In fact, the same moment bound holds for a much larger class of pinning measures $\rho$. However, for the more general case we do not have such a short and simple proof. 

\subsection{Duplantier-Miller-Sheffield's approach}\label{Description:DMS14}
We now briefly describe the approach to quantum surfaces initiated in \cite{Zipper,mating}. In particular we will show how to make sense of the unit area quantum sphere even in the case of two marked points.

The underlying idea is to think of quantum surfaces as abstract surfaces, and of different conformal parametrizations as of different embeddings of these surfaces. One way to do this is to define an equivalence class of random measures with a transformation rule under conformal mappings.

\begin{definition}\label{def:srf}
	Let $\gamma \in (0,2)$ and $Q = 2/\gamma + \gamma/2$. Suppose we have a random distribution $h$ on a domain $D$. Then a random surface with $k \geq 0$ marked points is an equivalence class of $(k+2)$-tuples $(D,h,x_1,\dots,x_k)$ (with $x_i\in\overline{D}$) under the following equivalence relation: two embeddings $(D_i, h_i, x^i_1,\dots,x^i_k)$ for $i = 1,2$ are considered equivalent if there is a conformal map $\phi:D_2 \rightarrow D_1$ such that
	\begin{enumerate}
		\item $h_2 = h_1\circ\phi+Q\log|\phi'|;$
		\item it induces a correspondence between marked points on $D_1$ and marked points on $D_2$: $x^1_j=\phi(x^2_j), j=1,\dots,k$.
	\end{enumerate}
\end{definition}
In each fixed embedding, one can define the corresponding Liouville measure $\mu_{h_1}$ (resp. $\mu_{h_2}$) as the Gaussian multiplicative chaos measure of parameter $\gamma$ corresponding to the field $h_1$ (resp. $h_2$). Here the transformation rule in the definition is chosen so that this $\gamma$-Liouville measure is defined intrinsically for any equivalence class: $\mu_{h_1}=\phi_{*}\mu_{h_2}$. The marked points in this framework correspond to the so called typical points of this Liouville measure, i.e. they have $\gamma$-singularities.

\begin{remark}
	If we have too few marked points to fix the automorphisms of $D$, then the sigma-algebra of $h$ does not coincide with that of the random surface -- i.e. only events that are invariant under the remaining automorphisms are measurable w.r.t the random surface.
\end{remark}
\begin{remark}
	Given a random surfaces with fewer marked points than are necessary to fix an embedding, we can still find well-defined embeddings of this surface: for example, we can just choose the missing number of points in a measurable way w.r.t. the random surface constructed with given marked points and then use all these points to choose a well-defined embedding.
\end{remark}

\subsubsection{Encoding surfaces using Bessel processes}\label{sec:Bessel}
In \cite{mating}, the authors introduce a way to encode random surfaces using Bessel processes and give a definition of the quantum unit sphere with  two marked points.

To motivate the construction, consider the infinite cylinder $\mathcal{Q}=\mathds{R}\times[0,2\pi]$ with $\mathds{R}\times \{0\}$ and $\mathds{R} \times \{2\pi\}$ identified. Here two marked points are given by the ends of the cylinder $\{-\infty,+\infty\}$ and the remaining degrees of Mobius freedom are: 1) rotations around the axis of the cylinder; 2) horizontal shifts. 

These degrees of freedom pair well with the decomposition of the GFF on $\mathcal{Q}$ into radial and angular parts \cite[Lemma 4.2]{mating}: the radial part is invariant under rotations and transforms under horizontal shifts while the angular part is invariant under horizontal shifts and transforms under rotations. This comes directly from the orthogonal decomposition of the Dirichlet space $\mathcal{H}(\mathcal{Q})$ on the cylinder, that is given by the closure of smooth functions $f$ on $\mathcal{Q}$ under the Dirichlet norm $\|\nabla f\|$ \cite[Lemma 4.2]{mating}:
\begin{lemma}\label{GFFDecompositionCylinder}
	Let $\mathcal{H}_1(\mathcal{Q}) \subset \mathcal{H}(\mathcal{Q})$ be the subspace of of functions which are constant on circles of the form $u\times[0,2\pi]$ with endpoints $0$ and $2\pi$ identified and $u\in\mathds{R}$. Also, let $\mathcal{H}_2(\mathcal{Q}) \subset \mathcal{H}(\mathcal{Q})$ be the subspace of functions on $\mathcal{Q}$ which are of mean zero on all such circles. Then $\mathcal{H}_1(\mathcal{Q})$ and $\mathcal{H}_2(\mathcal{Q})$ form an orthogonal decomposition of the space $\mathcal{H}(\mathcal{Q})$.
\end{lemma}

\begin{remark}
	We call the $\mathcal{H}_1(\mathcal{Q})$ and $\mathcal{H}_2(\mathcal{Q})$ components the radial and angular components respectively.
\end{remark}
Let $\delta = 4- 8/\gamma^2$ and let $\nu_\delta^{BES}$ be Bessel excursion measure of dimension $\delta$.  In order to define $\DMSsph{2}$,  we first introduce an infinite measure on  quantum surfaces with two marked points as follows:
\begin{enumerate}
	\item Parameterize quantum surfaces by $(h,\cQ,\infty,-\infty)$.
	\item Sample an excursion $e$ according to $\nu_\delta^{BES}$.
	\item The radial $\mathcal{H}_1(\mathcal{Q})$ component of the quantum sphere is given by reparametrizing $2\gamma^{-1} \log e$ to have quadratic variation $du$.
	\item The angular $\mathcal{H}_2(\mathcal{Q})$ component is sampled independently from the law of the $\mathcal{H}_2(\mathcal{Q})$ component of a whole-plane GFF on $\mathcal{Q}$.
\end{enumerate}
Following \cite{mating2}, we denote this measure by $M_{\mathrm{BES}}$.

\begin{definition}\label{def:DMS-sphere}
	The unit area quantum sphere with two marked points is $M_{\mathrm{BES}}$ conditioned on the surface having unit quantum area. We denote its law  by $\DMSsph{2}$. Given an integer $k\ge 3$ and if we first sample of instance of $\DMSsph{2}$ then conditionally independently  sample $k-2$ points on the surface according to the quantum area, the resulting quantum surface with $k$ marked points is called the unit area quantum sphere with $k$ marked points, whose law is denoted by $\DMSsph{k}$. By forgetting one marked point in $\DMSsph{2}$, we obtain the unit area quantum sphere with 1 marked point $\DMSsph{1}$.
\end{definition}
\begin{remark}
	Although $M_{\mathrm{BES}}$ is an infinite measure, the restriction of $M_{\mathrm{BES}}$ to the event that the surface has quantum area larger than a positive number is finite.  One can check that it is also similarly possible to make sense of $M_{\mathrm{BES}}$ conditioned on have quantum area 1 to obtain a probability measure. 
\end{remark}
Notice that a priori $\DMSsph{2}$ having only two marked points, comes without a canonical embedding. Choosing an embedding amounts to fixing the rotation and the horizontal translation in some way. The rotation can be fixed arbitrarily. For the radial part, we find it convenient to use the so called \emph{maxima embedding}, where we fix the horizontal shift by requiring the location of the radial maxima to be at zero.

\subsubsection{Limiting procedure for the 2-point sphere}
In \cite[Proposition 5.13]{mating}, a limiting procedure is described to obtain $\DMSsph{2}$. It roughly goes as follows: take a Dirichlet GFF $h$ on the unit disk $\D$ with zero boundary conditions, and condition its Gaussian multiplicative chaos measure $\mu_h$ to have mass $e^{C} \leq \mu_h(\D) \leq e^C(1+\delta)$ for $C > 0$ large and $\delta > 0$ small. Denote the conditional law by $h_{C, \delta}$. Let $w$ be a point in $\D$ sampled from $\bar\mu_{h_{C,\delta}}$. It is then shown that the quantum surfaces $(h_{C,\delta} - \gamma^{-1}C,\D,w, \infty)$ converge in law to $\DMSsph{2}$, when we first let $C \to \infty$ and then let $\delta \to 0$.
	
	Here, convergence means that in some fixed embedding of these quantum surfaces the conditioned $\gamma$-Liouville measures converge weakly in law to the measure corresponding to $\DMSsph{2}$ in this embedding. For example, one can 
	\begin{itemize}
		\item map $\C$ to the cylinder $\mathcal{Q}$ via the logarithmic map, sending $\infty$ to $-\infty$ and $w$ to $\infty$;
		\item use the change-of-coordinates formula of Definition \ref{def:srf} to transform the field;
		\item horizontally shift the resulting field so that the maxima of its radial part is achieved at $\Re z = 0$;
		\item consider the $\gamma$-Liouville measure of the resulting field and, in order to obtain a measure on $\cQ$, extend this measure by zero outside of the shifted image of $\D$. 
	\end{itemize}
	Convergence of $(h_{C,\delta} - \gamma^{-1}C,\D, w,\infty)$ then means convergence in law of these measures on the cylinder $\mathcal{Q}$ under weak topology. The two marked points $w$ and $\infty$ become respectively $\infty$ and $-\infty$ on $\cQ$.
	
	By examining the corresponding proof in \cite{mating}, we lift the following statement:

\begin{proposition}\label{prop:2ptlim}
	Let $h=h_0-\gamma \log|z|-C$ where $h_0$ is a zero boundary GFF on $\D$ and $C$ is a constant. The quantum surface $(h,\C,0,\infty)$  conditioned on $e^{-\gamma\delta} \leq \mu_h(\D) \leq e^{\gamma\delta}$  converges to $\DMSsph{2}$. More precisely, if we embed the quantum surface $(h,\C,0,\infty)$ into $(\cQ,+\infty,-\infty)$ and fix the horizontal shift such that the maxima of the radial part is achieved at $\Re z=0$, then as $C\to \infty$ and then $\delta \to 0$ the conditional law of the quantum surface converges to the maxima embedding of $\DMSsph{2}$ described in Section~\ref{sec:Bessel} in the sense the corresponding $\gamma$-Liouville measures converge weakly in law.
\end{proposition}

\section{A limiting procedure for \texorpdfstring{$\DKRVsph{\gamma,\gamma,\gamma}$}{DKRV}}\label{sec:3pt}
This section is devoted to proving an approximation scheme for $\DKRVsph{\gamma,\gamma,\gamma}$. Let $D^\eps=\eps^{-1}\D$ be a large disk and $h^\eps_0$ be the zero boundary GFF on $D^\eps$. Let further $z_1,z_2\in \D$ and consider a GFF with two interior $\gamma$ singularities at $z_1, z_2$, and boundary values chosen such that in the limit they give rise to a third $\gamma$ singularity at infinity, i.e. we set:
\begin{equation}
\label{def:h}
h^\eps(z)=h_0^\eps(z)+  (2Q-\gamma)\log\eps + \gamma G_{D^\eps}(z,z_1)+\gamma G_{D^\eps}(z,z_2).
\end{equation}
Here $G_{D^\eps}$ is the Dirichlet Green function of $D^\eps$. 

Let $\P^\eps$ be the corresponding probability measure and $\mu_{h^\eps}$ be the associated Liouville measure of $h^\eps$. Furthermore, let $A_\eps =(h_0^\eps, \xi_\eps)_\nabla$ where $\xi_\eps = -\log(|z|\vee 1)-\log\eps$. Then $A_\eps$ is the circle average of $h^\eps_0$ around $\partial \D$.
\begin{theorem}\label{thm:3pt}
	In the above setup, suppose $\delta>0$, $E^\eps_\delta=\{\mu_{h^\eps }(\C)\in[e^{-\gamma\delta},e^{\gamma\delta}] \}$ and $H^\eps = \{ A_\eps  + (2Q-3\gamma) \log \eps \geq -|\log \eps|^{2/3}\}$. By first letting $\eps\to 0$ then $\delta\to 0$, $\mu_{h^\eps}$ conditioned on $E^\eps_\delta \cap H^\eps$ converges in law to $\DKRVsph{\gamma,\gamma,\gamma}$ of Definition \ref{def:DKRV-def} and with singularities at $z_1, z_2, \infty$. Here the topology of convergence is the weak topology of measures on $\C\cup\{\infty\}$. 
\end{theorem}

In the case $\gamma \geq \sqrt{2}$, the proof of Lemma~\ref{lem:Gaussian2} below implies that one can omit conditioning on $H^\eps$ and the theorem still holds. See Remark~\ref{rmk:Heps} for details. The proof of Proposition~\ref{prop:3pt} in Section~\ref{sec:equi} also suggests that the same is true when $\gamma<\sqrt{2}$. However, we cannot confirm this via a short argument for now.

\subsection{Preliminary calculations, notations and heuristics}\label{subsec:preliminary}
Using conformal invariance of the Green's function, we can write
\[h^\eps = h^\eps_0 + (2Q-3\gamma)\log \eps - \gamma \log |z-z_1| - \gamma \log |z-z_2| + r_\eps, \]
where $r_\eps =\gamma \log |1-\eps^2 \bar z_1 z|+\gamma \log |1-\eps^2 \bar z_2 z|$.

Recall that $A_\eps$ is the circle average of $h^\eps_0$ around $\partial \D$ and $\Var[A_\eps] = (\xi_\eps, \xi_\eps)_\nabla = - \log \eps$.
By the orthogonal decomposition of the GFF (see e.g. \cite[Section 2.4]{GFF}), we can write  
\[h_0^\eps = h^\eps_{\mathfrak c} + \frac{A_\eps}{\Var[A_\eps]}\xi_\eps,\]
where $h^\eps_{\mathfrak c}$ is distributed as $h^\eps_0$ conditioned to have circle-average zero on $\partial\D$ and is independent of $A_\eps$.  

Let $h_{\mathfrak c}$ be the whole-plane GFF normalized such that the circle average around $\partial \D$ is zero, i.e. in other words we consider the scalar whole-plane GFF with the background measure equal to the uniform measure on the unit circle. For consistence of notation we denote this unit measure by $\mathfrak c$. Then using the notations of Section 2.2, we write
\begin{align}\label{eps}
h^\eps_L& = h^\eps_{\mathfrak c} - (2Q-3\gamma)\log (|z| \vee 1) - \gamma \log |z-z_1| - \gamma \log |z-z_2|,\\
h_L &= h_{\mathfrak c} - (2Q-3\gamma)\log (|z| \vee 1) - \gamma \log |z-z_1| - \gamma \log |z-z_2|.
\end{align}
The corresponding probability laws are denoted by $\P_{\mathfrak c}^\eps,\P_{\mathfrak c}$. The field $h_{\mathfrak c}$ now corresponds to the field of \eqref{eq:exactdefoffield}, in the case where the background measure is the unit mass distributed uniformly on the unit circle. By Proposition~\ref{prop:whole-plane}, $\P_{\mathfrak c}$ is the limiting measure of $\P^\eps_{\mathfrak c}$, where the topology is given by convergence in total variation on any bounded domain.

We can now write
\begin{equation}
\label{def:h2}
h^\eps = h^\eps_L + \left(\frac{A_\eps}{\Var[A_\eps]} - (2Q-3\gamma)\right) \xi_\eps + r_\eps
\end{equation}
and by the orthogonal decomposition we have $d\P_\eps = dA_\eps\otimes d\P_{\mathfrak c}^\eps$. 
The key observation is that if we now consider the probability measure $$d\Q^\eps \propto \exp\{(2Q-3\gamma)A_\eps\}d\P^\eps,$$then under $\Q^\eps$:
\begin{enumerate}
	\item by Girsanov theorem, the law of  $\wt A_\eps:=A_\eps+(2Q-3\gamma)\log\eps$ is a centered Gaussian with variance $|\log \eps|$;
	\item $h^\eps_L$ has the same law as under $\P^\eps$ and is independent of $\wt A_\eps$.
\end{enumerate}
Now write 
\begin{equation}\label{eq:underQ}
h^\eps =h^\eps_L + \wt A_\eps + \wt A_\eps g_\eps(z) + r_\eps,
\end{equation}where $g_\eps(z)=\log(|z|\vee 1)/\log\eps$. 

Writing $h^\eps$ in this way is useful, as under $\Q^\eps$, $\wt A_\eps g_\eps(z) + r_\eps$ vanish as $\eps$ goes to 0 and therefore $h^\eps$ is roughly speaking an independent sum of $h^\eps_L$ and $ \wt A_\eps$. Since $h^\eps_L$ tends to $h_L$ in law and the law of $ \wt A_\eps$ vaguely tends to Lebesgue measure, this should remind the reader of the whole-plane GFF tensorized with Lebesgue measure, as discussed in Section~\ref{subsec:whole} and used in \cite{DKRV}.  Making this connection precise is the main content of the rest of this section.

\subsection{Statement of the main lemma and proof of Theorem~\ref{thm:3pt}}
In the rest of this section we omit $r_\eps$, as it is a harmonic function with $\max_z| r_\eps(z)|=o_\eps(1)$ and plays no role in the later argument. 
Given a Liouville measure $\mu_h$ with almost surely finite total mass, let $\bar\mu_{h}(\cdot)=\mu_{h}(\cdot)/\mu_{h}(\C)$.  
Recall the events $\E^\eps_\delta$ and $H^\eps$ from Theorem~\ref{thm:3pt}:
\[E_\delta^\eps=\{\mu_{h^\eps}(\C)\in[e^{-\gamma\delta},e^{\gamma\delta}]\}.\]
In light of \eqref{eq:underQ}, under $\Q^\eps$ it is instrumental to write
\[
H^\eps = \{\wt A_\eps  \geq -|\log \eps|^{2/3}\}.
\] 
Now Theorem~\ref{thm:3pt} is a consequence of the following lemma:
\begin{lemma}\label{lem:Gaussian2} Suppose $F$ is a nonnegative bounded continuous functional on the space of probability measures on $\C\cup\{\infty\}$ with the topology of weak convergence and $a=\gamma^{-1}(2Q-3\gamma)$. Then for all $\delta>0$
	\begin{equation}
	\label{eq:Gaussian2}
	\lim_{\eps\to 0} \sqrt{2\pi \Var[\wt A_\eps]}\E^{\Q^\eps}[F(\bar \mu_{h^\eps_L+\wt A_\eps g_\eps(z)}) \mu_{h^\eps_L+\wt A_\eps g_\eps}^{a}(\C) \1_{E^\eps_\delta}\1_{H^\eps}]=2\delta\E^{\P_{\mathfrak c}}[F(\bar \mu_{h_L}) \mu_{h_L}^{a}(\C)].
	\end{equation}
\end{lemma} 

We postpone the proof of Lemma~\ref{lem:Gaussian2} in Section~\ref{subsec:key} and proceed to the proof of Theorem~\ref{thm:3pt}.
\begin{proof}[Proof of Theorem~\ref{thm:3pt}]
	Let $\P_{\mathfrak c}'$ be the probability measure we obtained via re-weighting $\P_{\mathfrak c}$ by $\mu_{h_L}^{a}(\C)$. Then by Theorem~\ref{thm:iofm} and Definition~\ref{def:DKRV-def}, the law of $\DKRVsph{\gamma,\gamma,\gamma}$ is given by $\bar \mu_{h_L}$ under $\P_{\mathfrak c}'$. 
	
	To prove Theorem~\ref{thm:3pt}, we first notice that since we are conditioning on $\mu_{h^\eps}$ having total mass between $[e^{-\gamma\delta},e^{\gamma\delta}]$ and we are letting $\delta\to 0$ eventually, it suffices to show that $\bar \mu_{h^\eps}$ converges to  $\bar\mu_{h_L}$ under $\P'_{\mathfrak c}$. In other words it suffices to show that for all $F$ as in Lemma~\ref{lem:Gaussian2} we have
	
	\[\lim_{\delta\to 0}\lim_{\eps\to 0}\E^{\P^\eps}[F(\bar \mu_{h^\eps})|E^\eps_\delta\cap H^\eps]=\E^{\P'_{\mathfrak c}}[F(\bar \mu_{h_L})].\]
	
	Recall that $a=\gamma^{-1}(2Q-3\gamma)$. We have
	\begin{equation}\label{eq:measure-change}
	\lim_{\delta\to 0}\lim_{\eps\to 0}\frac{ \E^{\P^\eps}[F(\bar\mu_{h^\eps})  \1_{E^\eps_\delta}\1_{H^\eps}] }{ \E^{\P^{\eps}}[E^\eps_\delta\cap H^\eps]}=\lim_{\delta\to 0}\lim_{\eps\to 0}\frac{ \E^{\Q^\eps}[F(\bar\mu_{h^\eps})  e^{-(2Q-3\gamma)A_\eps} \1_{E^\eps_\delta}\1_{H^\eps}] }{ \E^{\Q^\eps}[ e^{-(2Q-3\gamma) A_\eps}  \1_{E^\eps_\delta}\1_{H^\eps}]}.
	\end{equation}
	Now under the event $E^\eps_\delta$ we have
	\[
	e^{-(2Q-3\gamma) A_\eps}\propto e^{-(2Q-3\gamma)\wt A_\eps} =(1+O(\delta))\mu_{h^\eps_L+\wt A_\eps g_\eps(z)}^{a} (\C).
	\]
	Therefore the right hand side of \eqref{eq:measure-change} equals
	\begin{equation*}
	\lim_{\delta\to 0}\lim_{\eps\to 0}\frac{ \E^{\Q^\eps}[F(\bar\mu_{h^\eps_L+\wt A_\eps g_\eps(z)}) \mu_{h^\eps_L+\wt A_\eps g_\eps(z)}^{a} (\C) \1_{E^\eps_\delta}\1_{H^\eps}] }{ \E^{\Q^\eps}[ \mu_{h^\eps_L+\wt A_\eps g_\eps(z)}^{a} (\C) \1_{E^\eps_\delta}\1_{H^\eps}]}.
	\end{equation*}
	By Lemma~\ref{lem:Gaussian2}, it is equal to
	\[\frac{ \E^{\P_{\mathfrak c}}[F(\bar \mu_{h_L}) \mu_{h_L}^{a}(\C)]}{ \E^{\P_{\mathfrak c}}[ \mu_{h_L}^{a}(\C)]}= \E^{\P'_{\mathfrak c}}[F(\bar\mu_{h_L})],
	\]
	from which Theorem~\ref{thm:3pt} follows. 
\end{proof}

\subsection{Proof of Lemma \ref{lem:Gaussian2}}\label{subsec:key}
The proof of Lemma \ref{lem:Gaussian2} relies on two simple Fubini identities and a continuity result.
\begin{lemma}\label{lem:Fubini}
	Let $X,Y,Z$ be random variables where $\E[|Z|]<\infty$. For fixed $\delta>0$, let 
	$$E^1_\delta(x) = \{X \in [-x -\delta,\delta-x]\},E^2_\delta(x)=\{X\le-x+\delta,Y\ge-x-\delta\}.$$
	Then 
	\begin{align}
	\label{eq:F1}
	&\int_{-\infty}^\infty \E [Z1_{E^1_\delta(x)}]dx=2\delta \E[Z],\\
	\label{eq:F2}
	&\int_{-\infty}^\infty \E [Z1_{E^2_\delta(x)}]dx = \E[Z((Y-X+2\delta)\vee 0)].
	\end{align}
\end{lemma}
\begin{proof}
	\eqref{eq:F1} is the special case of \eqref{eq:F2} when $X=Y$. So we only prove \eqref{eq:F2}. By Fubini Theorem the left hand side of \eqref{eq:F2} equals \[
	\E\left[\int_{-\infty}^\infty Z1_{\{ -Y-\delta\le x\le-X+\delta \}}dx\right] =\E[Z((Y-X+2\delta)\vee 0)]. \qedhere
	\]
\end{proof}

\begin{lemma}\label{lem:continuity}
	Suppose $a_\eps$ is a sequence of deterministic numbers tending to 0 as $\eps\to 0$. Let $h_L^\eps,h_L$ be defined as in Section~\ref{subsec:preliminary}. Then $\bar\mu_{h_L^\eps+a_\eps\log(|z|\vee 1)}$ converges in law to $\bar \mu_{h_L}$ in the space probability measure on $\C\cup\{\infty \}$ endowed with the weak convergence
	topology. Moreover, $\E^{\P^\eps_{\mathfrak c}} \big[\mu^q_{h_L^\eps+a_\eps\log(|z|\vee 1)}(\C)\big]$ converges to $\E^{\P_{\mathfrak c}}\big[\mu^q_{h_L}(\C)\big]$ for any $q<\tfrac{2}{\gamma}(Q-\gamma)$.
\end{lemma}

\begin{proof}
		By Proposition~\ref{prop:whole-plane}, $h_L^\eps$  restricted to $R\D$ converges to $h_L$ in total variation for all $R>0$. Moreover, $a_\eps \log(|z|\vee 1)$ converges to 0 locally uniformly in $\C$. Therefore on $R\D$, the random measure $\mu_{h_L^\eps+a_\eps\log(|z|\vee 1)}$ converges in total variation distance to $\mu_{h_L}$. 
		
		We will now argue that on the one hand for all $q <\tfrac{2}{\gamma}(Q-\gamma)$ and for some constant $C>0$
		\begin{equation}\label{mb}
\limsup_{\eps\to 0}\,	\E^{\P^\eps_{\mathfrak c}}\left[\mu^q_{h_L^\eps+a_\eps\log(|z|\vee 1)} (\C)\right] < C,
		\end{equation}
		and on the other hand
		\begin{equation}\label{mb2}
		\lim_{R\to\infty}\limsup_{\eps\to\infty}	\E^{\P^\eps_{\mathfrak c}} \left[\mu^q_{h_L^\eps+a_\eps\log(|z|\vee 1)}(\C \backslash R\D)\right]=0.
		\end{equation}

		Combined with the fact that $\E^{\P_{\mathfrak c}}[\mu^q_{h_L} (\C \backslash R\D)]=o_R(1)$, this will prove the convergence of $\bar\mu_{h_L^\eps+a_\eps\log(|z|\vee 1)}$ weakly in law to $\bar\mu_{h_L}$ and also the convergence in $q$-th moment of $\mu_{h_L^\eps+a_\eps\log(|z|\vee 1)} (\C)$ for all $q<\tfrac{2}{\gamma}(Q-\gamma)$. 
				
		Notice that on $\D$, the covariance kernels of $h_{\mathfrak c}$ and $h^\eps_{\mathfrak c}$ differ by at most a constant. Thus by Corollary \ref{cor:moments} we see that $$\E^{\P^\eps_{\mathfrak c}}[\mu^q_{h_L^\eps+a_\eps\log(|z|\vee 1)} (\D)] < \infty $$ for all $q <\tfrac{2}{\gamma}(Q-\gamma)$. In particular, \eqref{mb} already follows for negative moments.
		
		Now fix any $0< q<\tfrac{2}{\gamma}(Q-\gamma)$ and take $c > 0$ small enough such that $q<\tfrac{2}{\gamma}(Q-\gamma-c)$.	
		Notice that	for small enough $\eps$
		$$\mu_{h_L^\eps+a_\eps\log(|z|\vee 1)}(\C \backslash \D)<\mu_{h_L^\eps+c\log(|z|\vee 1)}(\C \backslash \D).$$ Let $A_n$ denote the annulus $e^n\D\backslash e^{n-1}\D$.
		We claim that $$\E^{\P^\eps_{\mathfrak c}}[\mu_{h_L^\eps+c\log(|z|\vee 1)}^q(A_n)] < b_1e^{-nb_2}$$ for some $b_1,b_2 > 0$. This proves both equations \eqref{mb} and \eqref{mb2}. 
	
		To show this, consider the radial decomposition of $h^\eps_{\mathfrak c} = h_r + h_a$. The radial part $h_r(z)$, parametrized by $\log |z|$ has the law of a zero-to-zero Brownian bridge with end points at $0$ and $-\log \eps$. Thus using \eqref{eps}, we can bound 
		$\E^{\P^\eps_{\mathfrak c}} \left[\mu^q_{h_L^\eps+c\log(|z|\vee 1)}(A_n)\right]$ by
		$$O(1)\lim_{\delta \to 0} \E^{\P^\eps_{\mathfrak c}}  \left[\left(\int_{A_n} \frac{1}{|z|^{\gamma(Q-\gamma -c)}}\exp(\gamma h_r(z))\frac{1}{|z|^{\gamma Q}} \delta^{\gamma^2/2} \exp(\gamma h_a^\delta(z)) dz\right)^q\right],$$
		where the writing suggests how we decompose radial and angular parts, and $O(1)$ absorbs the constant order difference $2\gamma \log |z| - \gamma \log |z_1-z| - \gamma \log |z_2-z|$.

		By using independence between the radial and the angular part,  bounding $h_r$ by $s_n = \sup_{|z| \in [e^{n-1},e^n]} h_r(z)$ and $|z|$ by $e^{n}$, we can further bound by
			$$O(1)\E[\exp(-nq\gamma(Q-\gamma-c) + q\gamma s_n) ]\lim_{\delta \to 0} \E^{\P^\eps_{\mathfrak c}} \left[\left( \int_{A_n} \frac{1}{|z|^{\gamma Q}} \delta^{\gamma^2/2} \exp(\gamma h_a^\delta(z)) dz\right)^q\right].$$
		Let us treat the first expectation on the radial part. We write $$s_n = h_r(e^{n-1}) + \sup_{|z| \in [e^{n-1},e^n]}	(h_r(z) -h_r(e^{n-1})).$$ The latter summand is just a Brownian piece of time duration 1, thus its exponential moments are uniformly bounded. For the former summand, it is a centred Gaussian whose variance is bounded from above by $n$. Thus from an explicit Gaussian calculation we get that:
		\[
		\E[\exp(-nq\gamma(Q-\gamma-c) + q\gamma s_n) ]\le O(1)\E [\exp(-nq\gamma(Q-\gamma-c) + q^2\gamma^2 n/2) ]
		\]	
		But now $-q\gamma(Q-\gamma-c) + q^2\gamma^2/2$ is exactly zero at $q=0$ and at $q = \tfrac{2}{\gamma}(Q - \gamma -c)$ and in particular it is negative for any $q<\tfrac{2}{\gamma}(Q-\gamma-c)$. 
		
		Finally, for the angular part one can use standard results of Gaussian multiplicative chaos theory: as the covariance kernel of the angular part in each $A_n$ writes $- \log |z-w| + g(z,w)$ with $g(z,w)$ uniformly bounded in $n$, one can use Kahane's convexity inequality \cite[Theorem~2.1]{RV_Review} to see that for $q' < 4/\gamma^2$, the $q'$-th moment of the angular part is bounded uniformly in $n$ (e.g. see Lemma 5.4 in \cite{mating} for an almost equivalent statement and proof.). Now we obtain \eqref{mb} and \eqref{mb2} and conclude the proof.
\end{proof}

\begin{proof} [Proof of Lemma~\ref{lem:Gaussian2}]
	Denote the left hand side of \eqref{eq:Gaussian2} by $L_\eps$ and let 
	\[
	E^\eps_\delta(x)=\{ \gamma^{-1}\log\mu_{h^\eps_L+xg_\eps(z)}(\C)\in [-x-\delta,-x+\delta]\}.
	\]
	Then by conditioning on $\wt A_\eps=x$, we have
	\begin{equation}\label{I1}
	L_\eps=\int_{-|\log\eps|^{\frac23}}^\infty \E^{\P_{\mathfrak c}^\eps}[F(\bar\mu_{h_L^\eps+xg_\eps(z)}) \mu_{h^\eps_L+x g_\eps(z)}^{a}(\C) \1_{E^\eps_\delta(x)}] \exp\left\{-\frac{x^2}{2|\log\eps|}\right\}dx.
	\end{equation}
	Let
	\[
	\Sigma_\eps= \int 1_{x\in [-|\log\eps|^{\frac23},|\log\eps|^{\frac23}]} \E^{\P_{\mathfrak c}^\eps}[F(\bar\mu_{h_L^\eps+xg_\eps(z)}) \mu_{h^\eps_L+x g_\eps(z)}^{a}(\C) \1_{E^\eps_\delta(x)}] \exp\left\{-\frac{x^2}{2|\log\eps|}\right\}  dx
	\] 
	and $\Gamma_\eps=L_\eps-\Sigma_\eps$. 
	If $a\le 0$  (i.e. $\gamma\ge \sqrt{2}$), then 
	\[\mu_{h^\eps_L + x g_\eps(z)} (\C) \ge \mu_{h^\eps_L + x g_\eps(z)} (\D)=\mu_{h^\eps_L} (\D)\]
	and we have 
	\begin{equation}\label{eq:Gamma}
	\Gamma_\eps\le  O_F(1) \int_{x> |\log\eps|^{\frac23}} \E^{\P^\eps_{\mathfrak c}}[\mu^a_{h^\eps_L} (\D)]  \exp\left\{-\frac{x^2}{2|\log\eps|}\right\}dx=o_\eps(1).
	\end{equation}
	If $a>0$ (i.e. $\gamma<\sqrt{2}$), we get
	\[
	\Gamma_\eps\le  O_F(1) \int_{x> |\log\eps|^{\frac23}} e^{-\gamma ax}  \exp\left\{-\frac{x^2}{2|\log\eps|}\right\}  \P^\eps_{\mathfrak c}[E^\eps_\delta(x)] dx=o_\eps(1).
	\] Here $e^{-\gamma a x}$ comes from the fact that if $E^\eps_\delta(x)$ occurs and $\wt A_\eps=x$, then $\mu_{h^\eps_L+x g_\eps(z)}e^{\gamma x}=1+o_\delta(1)$. 

	It remains to deal with $\Sigma_\eps$. We first claim that the integrand converges pointwise, i.e. that for all fixed $x \in \R$ we have that 
		$$ 1_{x\in [-|\log\eps|^{\frac23},|\log\eps|^{\frac23}]} \E^{\P_{\mathfrak c}^\eps}[F(\bar\mu_{h_L^\eps+xg_\eps(z)}) \mu_{h^\eps_L+x g_\eps(z)}^{a}(\C) \1_{E^\eps_\delta(x)}] \exp\left\{-\frac{x^2}{2|\log\eps|}\right\}$$ converges to $$\E^{\P_{\mathfrak c}}[F(\bar\mu_{h_L}) \mu_{h_L}^{a}(\C) \1_{E_\delta(x)}].$$
		Taking $a_\eps=x/\log\eps$ in  Lemma~\ref{lem:continuity}, we have that $\bar\mu_{h_L^\eps+xg_\eps(z)}$ converges weakly in law to $\bar \mu_{h_L}$ and $\mu_{h^\eps_L+x g_\eps(z)}(\C)$ converges to $\mu_{h_L}(\C)$ in $L^q$ for $0<q<\tfrac{2}{\gamma}(Q-\gamma)$. Moreover, the argument of Lemma~\ref{lem:continuity} also gives joint convergence in law. Then using dominated convergence theorem, we have
		$$\lim_{\eps\to 0}\E^{\P_{\mathfrak c}^\eps}[F(\bar\mu_{h_L^\eps+xg_\eps(z)}) \mu_{h^\eps_L+x g_\eps(z)}^{a}(\C) \1_{E^\eps_\delta(x)}] = \E^{\P_{\mathfrak c}}[F(\bar\mu_{h_L}) \mu_{h_L}^{a}(\C) \1_{E_\delta(x)}]$$ and the claim follows.
	
	Now let $E_\delta(x)=\{ \gamma^{-1}\log\mu_{h_L}(\C)\in [-x-\delta,-x+\delta]\}$. Observe that by Lemma~\ref{lem:Fubini}
	\[
	\int_{-\infty}^{\infty}\E^{\P_{\mathfrak c}}[F(\bar\mu_{h_L}) \mu_{h_L}^{a}(\C) \1_{E_\delta(x)}]\ dx=2\delta\E^{\P_{\mathfrak c}}[F(\bar \mu_{h_L}) \mu_{h_L}^{a}(\C)]=
	\textrm{RHS of \eqref{eq:Gaussian2}}.
	\]
	Therefore to finish the proof we only need to show that the dominated convergence theorem can be applied to evaluate $\lim_{\eps\to 0} \Sigma_\eps$.
	
	When $a\le 0$, since $ g_\eps(z)=0$ on $\D$, the integrand inside $\Sigma_\eps$ is controlled by 
	\begin{equation}
	\label{eq:control}
	O_F(1)  \E^{\P_{\mathfrak c}^\eps}[\mu_{h^\eps_L}^{a}(\D) \1_{\hat E^\eps_\delta(x)}].
	\end{equation}
	By Lemma~\ref{lem:Fubini}
	\[
	\int^\infty_{-\infty}\E^{\P_{\mathfrak c}^\eps}[\mu_{h^\eps_L}^{a}(\D) \1_{\hat E^\eps_\delta(x)}]\ dx=\E^{\P_{\mathfrak c}^\eps}[\mu_{h^\eps_L}^{a}(\D)((Y_\eps-X_\eps+2\delta)\vee 0)].
	\]
	where 
	\begin{align*}
	& X_\eps=\gamma^{-1}\log\mu_{h^\eps_L+|\log\eps|^{\frac23}g_\eps(z)}(\C)\;
	\textrm{and}\; Y_\eps=\gamma^{-1}\log\mu_{h^\eps_L-|\log\eps|^{\frac23}g_\eps(z)}(\C),\\
	& \hat E^\eps_\delta(x)=\{ X_\eps \le-x+\delta,Y_\eps \ge -x-\delta\}.
	\end{align*}
	By Corollary~\ref{cor:moments}, for $c > 0$ small enough, we have that $\E^{\P_{\mathfrak c}}[\mu_{h_L + c\log(|z|\vee 1)}^{a+c}] < \infty$. Thus by the dominated convergence theorem,
	\[
	\lim_{\eps\to 0}\E^{\P_{\mathfrak c}^\eps}[\mu_{h^\eps_L}^{a}(\D)((Y_\eps-X_\eps+2\delta)\vee 0)]=2\delta\E^{\P_{\mathfrak c}}[\mu_{h^L}^{a}(\D)].
	\] Since
	\[
	\int^\infty_{-\infty}\E^{\P_{\mathfrak c}}[\mu_{h_L}^{a}(\D) \1_{E_\delta(x)}] \ dx=2\delta\E^{\P_{\mathfrak c}}[\mu_{h_L}^{a}(\D)],
	\]
	we have
	\[
	\lim_{\eps\to 0}\int^\infty_{-\infty}\E^{\P_{\mathfrak c}^\eps}[\mu_{h^\eps_L}^{a}(\D) \1_{\hat E^\eps_\delta(x)}]\ dx=\int^\infty_{-\infty}\E^{\P_{\mathfrak c}}[\lim_{\eps\to 0}\mu_{h^\eps_L}^{a}(\D) \1_{\hat E^\eps_\delta(x)}]\ dx.
	\]
	Now we can apply the dominated convergence theorem to evaluate $\lim_{\eps\to 0}\Sigma_\eps$.
	
	When $a>0$, (i.e. $\gamma<\sqrt{2}$), we replace the term in \eqref{eq:control} by 
	\begin{equation*}
	O_F(1)  \E^{\P_{\mathfrak c}^\eps}[\mu_{h^\eps_L+\log|\eps|^{\frac23} g_\eps(z) }^{a} (\C) \1_{\hat E^\eps_\delta(x)}].
	\end{equation*} 
	and the rest of the argument works line by line.
\end{proof}

\begin{remark}\label{rmk:Heps}
	Notice that when $a\le 0$ (i.e. $\gamma \geq \sqrt{2}$) the estimate in \eqref{eq:Gamma} still holds if we integrate against the whole $\R$ instead of $[-|\log\eps|^{2/3},\infty]$. This means that in Lemma~\ref{lem:Gaussian2} and hence also in Theorem~\ref{thm:3pt}, we can replace $E_\delta^\eps\cap H^\eps$ by $E_\delta^\eps$ if $\gamma\in[\sqrt 2,2)$. This does not directly work if $\gamma<\sqrt 2$, since $a$ is now positive and the integral over $[-\infty, -|\log\eps|^{2/3}]$ is not as easily controlled.
\end{remark}

\section{Equivalence of 3-point spheres}\label{sec:equi}
In this section we finish the proof of Theorem~\ref{thm:main}. By the Mobius invariance of $\DKRVsph{\gamma,\gamma,\gamma}$  in Theorem~\ref{thm:coordinate-change} and the relation between $\DMSsph{2}$ and $\DMSsph{3}$ in Definition~\ref{def:DMS-sphere}, it suffices to show that
\begin{theorem}\label{thm:main'}
	$\DKRVsph{\gamma,\gamma,\gamma}$ with marked points at $0, 1, \infty$ can be obtained by first sampling a third point $w$ from $\DMSsph{2}$  with marked points at $0,\infty$ and then applying the Mobius transformation that maps $(0,\infty,w)$ to $ (0,\infty,1)$.
\end{theorem}

We prove Theorem~\ref{thm:main'} by constructing two very close limiting procedures. The first one gives  $\DMSsph{3}$ embedded such that the marked points are $0,1,\infty$. The second one gives $\DKRVsph{\gamma,\gamma,\gamma}$ with marked points $0,1,\infty$. 

Let $D^\eps,h^\eps_0,G_{D^\eps}$ be defined as in Section~\ref{sec:3pt} and
\begin{equation}
\label{eq:h2pt}
h^\eps(z)= h^\eps_0(z)+(2Q-\gamma)\log\eps+\gamma G_{D^\eps}(z,0).
\end{equation} Notice that $h^\eps$ here is \emph{not} the same as the one defined in Section~\ref{sec:3pt}, indeed here $h^\eps$ only has a single log singularity inside $D^\eps$.

Denote $E^\eps_\delta=\{\mu_{h^\eps}(\C)\in [e^{-\gamma\delta},e^{\gamma\delta}] \}$, let $\P^\eps_\delta$ be the law of $h^\eps$ conditioning on $E^\eps_\delta$ and $d\hat \P^\eps_\delta=c \mu_{h^\eps(\C)} d\P^\eps_\delta$ where $c$ is chosen to make $\hat \P^\eps_\delta$ a probability measure. Given an instance of $h^\eps$ under $\P^\eps_\delta$, let $w^\eps$ be a point sampled according to the quantum area. Let $\hat w^\eps$ be sampled in the same way as $w^\eps$ with $\hat \P^\eps_\delta$ in place of $\P^\eps_\delta$. Now we claim that 
\begin{proposition}\label{prop:2pt}
	Suppose $(w^\eps,h^\eps)$ is sampled from $\P^\eps_\delta$ as above.  Apply the Mobius transformation that maps $w^\eps$ to 1 and fixes $0$ and $\infty$.  Let $h^\eps$ transform according to the coordinate change rule as in Definition \ref{def:srf}. If we first let $\eps\to 0$ then $\delta\to 0$, then   the resulting Liouville area measure converges weakly in law to that of $\DMSsph{3}$ with the embedding chosen such that the marked points are $0,1,\infty$ .
\end{proposition}
\begin{proof}
	Since $G_{D^\eps}(z,0)=-\log|\eps z|$, the field in \eqref{eq:h2pt} can be written as $$h^\eps=h^\eps_0-\gamma\log|z| + (2Q-2\gamma)\log\eps.$$ We now set the constant $C$ in Proposition~\ref{prop:2ptlim} to be $(2\gamma-3Q)\log\eps$ and apply the coordinate change $z\mapsto \eps z$, then Proposition~\ref{prop:2ptlim} yields that $(h^\eps,\C,0,\infty)$ converges to $\DMSsph{2}$ in the sense that in the maxima embedding, the area measures converge weakly. Therefore we have a coupling of $(h^\eps,\C,0,\infty)$, $\DMSsph{2}$, and points $w^\eps$ and $w$ sampled from $\mu_{h^\eps}$ and $\DMSsph{2}$ respectively, such that under maximal embedding $\mu_{h^\eps}$ a.s. converges to $\DMSsph{2}$ and $\lim_{\eps\to 0}|w^\eps-w|=0$ in probability. Now Proposition~\ref{prop:2pt} follows.
\end{proof}
We also have an analogous statement for $\DKRVsph{\gamma,\gamma,\gamma}$:
\begin{proposition}\label{prop:3pt}
	Suppose $(\hat w^\eps,h^\eps)$ is sampled from $\hat\P^\eps_\delta$ defined above.  Apply the Mobius transformation that maps $\hat w^\eps$ to 1 and fixes $0$ and $\infty$.  Let $h^\eps$ transform according to the coordinate change rule. If we first let $\eps\to 0$ then $\delta\to 0$, then the resulting Liouville area measure converges weakly in law to that of $\DKRVsph{\gamma,\gamma,\gamma}$ with  marked points  at $0,1,\infty$.
\end{proposition}

Let us see how Theorem~\ref{thm:main'} now follows. 

\begin{proof}[Proof of Theorem~\ref{thm:main'}]
	Since $\mu_{h^\eps}(\C)\in [e^{-\gamma\delta},e^{\gamma\delta}]$ under $\P^\eps_\delta$, the total variational distance between  $\P^\eps_\delta$ and  $\hat \P^\eps_\delta$ is $o_\delta(1)$ and so is the total variational distance between $(w^\eps,h^\eps)$ and $(\hat w^\eps,h^\eps)$. Therefore the two limiting objects in Proposition~\ref{prop:2pt} and Proposition~\ref{prop:3pt} are the same, which implies Theorem~\ref{thm:main'}. 
\end{proof}

The rest of the section is devoted to proving Proposition~\ref{prop:3pt}.

\subsection{The location of the third point sampled from \texorpdfstring{$\DMSsph{2}$}{DMS}}\label{subsec:2pt}
We start by controlling the location of the sampled points $w^\eps$ and $\hat w^\eps$ defined above. 
\begin{lemma}\label{lem:location}
	Let  $w^\eps$ be defined in Proposition~\ref{prop:2pt}. We have 
	$$\lim_{\delta\to 0}\lim_{\eps\to 0} \log |w^\eps|/|\log\eps|^{2/3}  =0\quad \textrm{in law under}\; \P^\eps_\delta.$$
	The same statement is true for $\hat w^\eps$  defined in Proposition~\ref{prop:3pt}.
\end{lemma}

\begin{proof}[Proof of Lemma~\ref{lem:location}]
	As the total variational distance between $(w^\eps,h^\eps)$ and $(\hat w^\eps,h^\eps)$ is $o_\delta(1)$, it suffices to prove the result for $w^\eps$.
	
	We use the map $z \mapsto e^{-z}$ to pull $(h^\eps,\C)$ back to $\cQ=\R\times[0,2\pi]$. Then $D^\eps$ is mapped to the half cylinder $(\log\eps,+\infty)\times[0,2\pi]$ and $0,\infty,w^\eps$ are mapped to  $+\infty,-\infty,-\log w^\eps$ respectively. The radial component of $h^\eps$ can be written as $X_t=B_t-(Q-\gamma)t$ where $B_t$ evolves as a standard Brownian motion with $B_{\log\eps}=2(Q-\gamma)\log\eps$. 
	
	Let $L^\eps$ be the location where $B_t-(Q-\gamma)t$ achieves the maxima. By the proof of Proposition~\ref{prop:2pt}, $L^\eps+\log |w^\eps|$ is tight, in particular
	\begin{equation}\label{eq:maxima}
	\lim_{\delta\to 0}\lim_{\eps\to 0}(\log|w^\eps|+L^\eps)/|\log\eps|^{2/3}=0\quad 	 \textrm{in law.}
	\end{equation}
	Thus we only need to show that
	\begin{equation}\label{eq:disp2}
	\lim_{\eps\to 0} L^\eps/|\log\eps|^{2/3}=0\quad 	 \textrm{in law}.
	\end{equation}
	In fact it suffices to show this under a certain conditioning. Indeed, let $F^\eps_C$ be the event that the maxima of $B_t-(Q-\gamma)t$ is bigger than $-C$. Then by the argument in \cite[Lemma~5.5 and Proposition~5.7]{mating},
	\begin{align}\label{eq:tight}
	&\P[F^\eps_C| E^\eps_\delta] \to 1 &&\textrm{as}\; C\to \infty\; \textrm{uniformly in}\; \eps\; \textrm{for fixed $\delta$ and}\\
	&\P[E^\eps_\delta|F^\eps_C ] >0  &&\textrm{uniformly in}\;  \eps\; \textrm{for}\; \textrm{fixed}\; \delta,C. \nonumber
	\end{align}
	Therefore we only need to show that \eqref{eq:disp2} holds for fixed $\delta,C$, conditioning on $F_C^\eps$.
	
	To prove this, let $T^\eps=\inf\{t\ge \log\eps: B_t-(Q-\gamma) t \ge -C \}$.  By the Markov property of Brownian motion, conditioning on $F^\eps_C$, $ L^\eps-T^\eps$ does not depend on $\eps, C, \delta$, and is a distribution associated with standard drifted Brownian motion. Thus it remains to show that for fixed $C$, conditioning on $F_C^\eps$,	
	\begin{equation*}
	\lim_{\eps\to 0} T^\eps/|\log\eps|^{2/3}=0\quad 	 \textrm{in law,}
	\end{equation*}
	which follows from the classical Lemma~\ref{lem:hitting} below by setting $a=Q-\gamma$ and properly translating the picture.
\end{proof}

\begin{lemma}
	\label{lem:hitting}
	Consider the drifted Brownian motion $B_t-at$ where $B_t$ is a standard Brownian motion with $B_t=0$ and $a>0$. Let $T_A=\inf\{t\in[0,\infty[:B_t-at=A\}$ with $A>0$. Then conditional on $T_A<\infty$, 
	\begin{equation}
	\lim_{A\to \infty} \frac{T_A-a^{-1}A}{A^{2/3}}=0 \quad\textrm{in law.}
	\end{equation}
\end{lemma}

\begin{proof}
	Recall \cite[Lemma 3.6]{mating}, conditioning on $T_A<\infty$, $B_t-at$ can be sampled as follows:
	\begin{enumerate}
		\item Sample a standard Brownian motion $X^1$  with linear drift $a$ until it hits $A$ and let $T_A$ be the first time that $X^1$ hits $A$.
		\item Sample  standard Brownian motion $X^2$ with linear drift $-a$.
		\item Concatenate $X^1[0,T_A] $ with $X^2(\cdot-T_A)$.
	\end{enumerate}
	It is well known (e.g. in \cite{krtshr}, also used in \cite{KPZ}) that the law of $T_A$ is the inverse Gaussian distribution with parameters $(a^{-1}A,a^{-2}A^2)$. Since $E[T_A]=a^{-1}A$ and $\Var[T_A]= aA$, Lemma~\ref{lem:hitting} follows from Markov inequality.
\end{proof}

\subsection{Rooted measure and \texorpdfstring{$\DKRVsph{\gamma,\gamma,\gamma}$}{DKRV}}\label{subsec:3pt}
In this section we prove Proposition~\ref{prop:3pt}. We will argue that the setting obtained by sampling a third point and mapping it to $1$ is close to that of Theorem~\ref{thm:3pt} with $z_1=0,z_2=1$, and that thus we can apply the arguments of Section 3 to obtain convergence to $\DKRVsph{\gamma,\gamma,\gamma}$.

Recall that $(\hat w^\eps,h^\eps)$ under $\hat\P^\delta_\eps$ can be sampled from the following probability measure:
\[c e^{\gamma h^{\eps}(z)}1_{E^\eps_\delta}dh^\eps dz\]
where $dh^\eps$ is the law of $h^\eps$ as in  \eqref{eq:h2pt} and $c$ is a normalizing constant.
We can perform this sampling in two steps:
\begin{enumerate}
	\item sample $(\hat w^\eps,h^\eps)$ from $c e^{\gamma h^{\eps}(z)}dh^\eps dz$, which is the so-called rooted measure of $h^\eps$;
	\item condition on the event $E^\eps_\delta$.
\end{enumerate}
Under the rooted measure $c e^{\gamma h^{\eps}(z)}dh^\eps dz$, $h^\eps$ can be written as (\cite[Section 3.3]{KPZ})
\begin{equation}\label{eq:root}
h^\eps= \tilde h^\eps_0+(2Q-\gamma)\log\eps+\gamma G_{D^\eps}(\cdot, \hat w^\eps) +\gamma G_{D^\eps}(\cdot,0)
\end{equation}
where $\hat w^\eps$ is sampled from its marginal law under the rooted measure
\footnote{
	As explained in \cite{KPZ}, the marginal law of $\hat w^\eps$ has density $\propto \int e^{\gamma h^\eps(z)}dh^\eps$ where the integration is over the law of $h^\eps$.
} 
and $\tilde h^\eps_0$ is an zero boundary GFF on $D^\eps$ that is independent of $\hat w^\eps$.

As stated in Proposition~\ref{prop:3pt}, we apply the Mobius transform $z\mapsto (\hat w^\eps)^{-1} z$ to pull $(0,\hat w^\eps,\infty)$ back to $(0,1,\infty)$. Let $\hat D^\eps=|\hat w^\eps|^{-1}D^\eps$, then the resulting field after coordinate change is  
\begin{equation}\label{eq:hath}
\hat h^\eps= \hat h^\eps_0+(2Q-\gamma)\log\eps+\gamma G_{\hat D^\eps}(\cdot, 1) +\gamma G_{\hat D^\eps}(\cdot,0) +Q\log|\hat w^\eps|
\end{equation}
where $\hat h^\eps_0$ denotes the zero boundary GFF on $\hat D^\eps$. Moreover, the event $E^\eps_\delta$ becomes $\mu_{\hat h^\eps}(\C)\in [e^{-\gamma\delta},e^{\gamma\delta}]$, and thus by conditioning $\hat h_\eps$ on the event $E^\eps_\delta$, we obtain the field in the statement in Proposition~\ref{prop:3pt}. 

In other words, if we first sample $(\hat w^\eps,h^\eps)$ according to $\hat\P^\eps_\delta$, then map $(0,\hat w^\eps,\infty)$ to $(0,1,\infty)$, then the resulting field after coordinate change of $h^\eps$ is exactly $\hat h^\eps$ on $\hat D^\eps$ as in \eqref{eq:hath} conditioning on $\mu_{\hat h^\eps}(\C)\in [e^{-\gamma\delta},e^{\gamma\delta}]$. 

\begin{proof}[Proof of Proposition~\ref{prop:3pt}]
	
	By Lemma~\ref{lem:location}, $\log |\hat w^\eps|=o_\eps(|\log\eps|^{\frac{2}{3}})$ with probability $1-o_{\eps, \delta}(1)$ under $\hat \P^\eps_\delta$. Here and henceforth we use the notation $o_{\eps, \delta}(1)$ to denote a quantity such that $\lim_{\delta\to 0}\lim_{\eps\to 0}o_{\eps,\delta}(1)=0$.
	
	Let $|\hat D^\eps|$ be the radius of $\hat D^\eps$. By Lemma~\ref{lem:location}, we can write 
	\begin{equation}
	\label{eq:rescale}
	\hat h^\eps =\hat h^\eps_0-(2Q-\gamma+q_\eps)\log|\hat D^\eps|  +\gamma G_{\hat D^\eps}(\cdot, 1) +\gamma G_{\hat D^\eps}(\cdot,0)
	\end{equation}
	where under $\hat \P^\eps_\delta$,
	\begin{equation}\label{eq:highprob}
	|\hat D^\eps|\to \infty\;\textrm{and}\; \lim_{\delta\to 0} \lim_{\eps\to 0} \frac{q_\eps}{|\log\eps|^{-1/3}}=0 \quad \textrm{in law}.
	\end{equation}
	Now let $\hat A_\eps$ be the circle average of $\hat h^\eps_0$ along $\partial \D$ and
	\[\hat H^\eps = \{\hat A_\eps +  (2Q-3\gamma+q_\eps) \log \eps \geq -|\log \eps|^{2/3}\}.\] 
	We claim that \begin{equation}\label{circ}
	\hat \P^\eps_\delta[\hat H^\eps ] = 1 - o_{\eps,\delta}(1).
	\end{equation}
	We will prove \eqref{circ} as a corollary of Lemma~\ref{lem:cond} below. Let us now see how it implies Proposition~\ref{prop:3pt}. Given \eqref{circ}, to study the limiting measure in Proposition~\ref{prop:3pt} as $\eps\to 0$ and then $\delta\to 0$, it suffices to study the limiting measure obtained by the following procedure:
	\begin{enumerate}
		\item[Step 1] sample a pair of random variables $(\hat D^\eps,q_\eps)$ such that as $\eps\to 0,\delta\to 0$, it holds that $|\hat D^\eps|\to \infty$ and $ q_\eps = o\left(\big|\log |\hat D^\eps|\big|^{-1/3}\right)$ in law. (Notice that in particular the pair  obtained from the marginal law on $\hat w^\eps$ as above satisfies these conditions);
		\item[Step 2] given $(\hat D^\eps,q_\eps)$, sample a zero boundary GFF on $\hat D^\eps $ (denoted by $\hat h^\eps_0$) and construct $\hat h^\eps$ by \eqref{eq:rescale}. Condition on both  $\mu_{\hat h^\eps(\C)}\in [e^{-\gamma \delta},e^{\gamma \delta}]$ and $\hat H^\eps$ occuring.
	\end{enumerate}
	In fact, although the measure in Proposition~\ref{prop:3pt} before taking the limit is not exactly as the one sampled from the above two steps, by \eqref{eq:highprob} and \eqref{circ}, the limits of these two procedures coincide in law.
	
	Now we are ready to use the proof of Theorem~\ref{thm:3pt} for the case $z_1=0,z_2=1$. First, notice that the sampling of the Step 2 is independent of Step 1 given $\hat D^\eps$. Thus to obtain the limiting law in Proposition~\ref{prop:3pt} we can simply assume that $q_\eps$ is a deterministic vanishing sequence and $\hat D^\eps$ is a sequence of deterministic growing disks. Now suppose that we replace $(2Q-\gamma)$ by $(2Q-\gamma+b_\eps)$ in \eqref{def:h} where $b_\eps$ is a deterministic sequence converging to 0. Then one can step-by-step examine that in Section~\ref{sec:3pt}, if whenever encountering $(2Q-3\gamma)$ we replace it by $(2Q-3\gamma+b_\eps)$, the same arguments still work. In particular, the conclusion of Theorem~\ref{thm:3pt} still holds, thus Proposition~\ref{prop:3pt} follows.\qedhere
\end{proof}

\bigskip

To finish off, we prove the claim \eqref{circ}. By examining \eqref{eq:hath}, \eqref{eq:rescale} along with \eqref{eq:highprob}, we see that with probability $1-o_{\eps,\delta}(1)$ the difference between $\hat A_\eps+(2Q-3\gamma)\log\eps$ and the circle average of $\hat h^\eps$ along $\partial \D$ is $o(|\log\eps|^{2/3})$. Moreover, the circle average of $\hat h^\eps$ along $\partial \D$ and the circle average of $ h^\eps$ along $|\hat w^\eps|\partial \D$ differ by $Q\log|\hat w^\eps|=o(|\log\eps|^{2/3})$ again with probability $1-o_{\eps,\delta}(1)$. Finally, since the total variational distance of $\hat \P^\eps_\delta$ and $\P^\eps_\delta$ is $o_\delta(1)$, we have that $\hat \P^\eps_\delta[\hat H^\eps]\ge \P^\eps_\delta[\wt H^\eps]-o_{\eps,\delta}(1)$, where $\wt H^\eps$ is the event that the circle average of $ h^\eps$ along $|w^\eps|\partial \D$ is larger than $-|\log\eps|^{2/3}$. Thus \eqref{circ} follows from:
\begin{lemma}\label{lem:cond}
	$\lim_{\delta\to 0}\lim_{\eps\to 0}\P^\eps_\delta[\wt H^\eps ] = 1$.
\end{lemma}
\begin{proof}
	Recall the setting of the proof of Lemma~\ref{lem:location}. In the cylindrical coordinates, the radial component of $h^\eps$ can be written as $X^\eps_t=B_t-(Q-\gamma)t$ such that $B_t$ evolves as a standard Brownian motion and $B_{\log\eps}=2(Q-\gamma) \log\eps$. By the coordinate change formula and Lemma~\ref{lem:location},
	\begin{align}\label{eq:Xt}
	\P^\eps_\delta[\wt H^\eps]\ge \P^\eps_\delta\left[X^\eps_{-\log|w^\eps|}\ge - |\log\eps|^{2/3}/2\right]-o_{\eps,\delta}(1).
	\end{align}
	Let $L^\eps$ be the location where $X^\eps_t$ achieves its maxima and $F^\eps$ be the event that $ X^\eps_{L^\eps}\ge - |\log \eps|^{2/3}/4$. Then \eqref{eq:tight} in Lemma~\ref{lem:location} implies that:
	\begin{equation}\label{maxi}
	\P[F^\eps | E^\eps_\delta] = 1 - o_\eps(1)\quad \textrm{for fixed $\delta$}.
	\end{equation}
	But now the proof of Proposition~\ref{prop:2pt} implies that  $X^\eps_{\log|w^{\eps}|}- X^\eps_{L^\eps}$ is a tight sequence of random variables. Hence \eqref{eq:Xt} and \eqref{maxi} together yield Lemma~\ref{lem:cond}.
\end{proof}

\end{document}